\documentclass[11pt,a4paper]{amsart}
\usepackage{amsmath}
\usepackage{amsfonts}
\usepackage{amssymb}
\usepackage{graphicx}
\usepackage[utf8]{inputenc}
\usepackage{pdflscape}
\usepackage{float}
\usepackage{mathtools}
\usepackage{appendix}
\usepackage[colorlinks=true, linkcolor=red, citecolor=blue]{hyperref}
\usepackage{epsfig}
\usepackage{tikz}
\usepackage{mathrsfs}

\numberwithin{equation}{section}
\makeatletter
\@namedef{subjclassname@2010}{\textup{2020} Mathematics Subject Classification}
\makeatother

\newtheorem{theorem}{Theorem}[section]
\newtheorem{definition}{Definition}[section]
\newtheorem{lemma}{Lemma}[section]
\newtheorem{proposition}{Proposition}[section]
\newtheorem{corollary}{Corollary}[section]
\begin{document}
\title[Control of Timoshenko laminated beams]{On the controllability of laminated beams 
 with Venttsel-type boundary conditions}
	\author[Bautista]{George J. Bautista}
	\author[Capistrano--Filho]{Roberto de A. Capistrano--Filho*}
	\author[Limaco]{Juan Límaco}
	
	\address{Instituto de Matemática e Estatística, Universidade Federal Fluminense, Campus  Gragoatá, 24210-201, Niterói (RJ), Brazil.}
		\email{\url{geojbs25@gmail.com}}
		\email{\url{jlimaco@id.uff.br}}
	
	\address{Departamento de Matem\'atica, Universidade Federal de Pernambuco, S/N Cidade Universit\'aria, 50740-545, Recife (PE), Brazil}
	\email{\url{roberto.capistranofilho@ufpe.br}}

%	\address{Instituto de Matemática e Estatística, Universidade Federal Fluminense, Campus  Gragoatá, 24210-201, Niterói (RJ), Brazil.}
%		\email{\url{jlimaco@id.uff.br}}
%		
	\thanks{$^*$Corresponding author: \url{{roberto.capistranofilho@ufpe.br}}}
	
	\thanks{\textbf{Funding.} Bautista was partially supported by FAPERJ under the program PDS-2024, grant number SEI-0260003/019497/2024. Capistrano–Filho was partially supported by CAPES-COFECUB program grant number 88887.879175/2023-00, CNPq grant numbers 421573/2023-6 and 307808/2021-1, and Propesqi (UFPE). L\'imaco was partially supported by CNPq grant number 310860/2023-7 (Brazil).}
	
	\begin{abstract}		
This paper examines the boundary controllability of a Timoshenko laminated beam system subject to Venttsel-type boundary conditions. The study focuses on a novel configuration in which three controls are applied solely at the boundary of the beam. Controllability is established by deriving an appropriate observability inequality for the corresponding adjoint system, which is then employed within the framework of the duality method in the setup of the classical Hilbert uniqueness method (HUM) to achieve the control problem. The main contribution lies in the analysis of a system comprising three beams governed by dynamic Venttsel-type boundary conditions, as introduced by Venttsel in \cite{Venttsel}.	
\end{abstract}
	
	\date{\today}
	\maketitle
	
%	\tableofcontents
	
	\thispagestyle{empty}
	
	%***********************************************
	%\normalsize
	\section{Introduction}

\subsection{Model under consideration}  This paper focuses on the analysis of the solution behavior for a Timoshenko-type laminated beam model governed by the following system
\begin{equation}%
\begin{cases}
\rho w_{tt}+G(\psi-w_{x})_{x}=0,&\mbox{ in } \left(  0,L\right)
\times\mathbb{R}^{+},\\
I_{\rho}\left(  3S_{tt}-\psi_{tt}\right)  -D(3S_{xx}-\psi_{xx})-G(\psi
-w_{x})=0, &\mbox{ in }\quad\left(  0,L\right)  \times\mathbb{R}^{+},\\
3I_{\rho}S_{tt}-3DS_{xx}+3G(\psi-w_{x})+4\delta_{0}S+4\gamma_{0}S_{t}%
=0,&\mbox{ in }\quad\left(  0,L\right)  \times\mathbb{R}^{+},
\label{eq1}
\end{cases}
\end{equation}
subject to homogeneous dynamic Venttsel type boundary conditions and prescribed initial data. Here, $L$ denotes the length of the beam, and the subscripts $x$ and $t$ indicate partial derivatives with respect to the spatial and temporal variables, respectively. In equation \eqref{eq1}, $w(x,t)$ denotes the transverse displacement, and $\psi(x,t)$ represents the angle of rotation. The function $S(x,t)$ corresponds to the relative slip along the adhesive interface at position $x$ and time $t$. The parameters $\rho$, $G$, $I_{\rho}$, $D$, $\delta_0$, and $\gamma_0$ are all positive constants, representing the density, shear stiffness, mass moment of inertia, flexural rigidity, adhesive stiffness, and adhesive structural damping coefficient, respectively.

If we change the variables $s=-3S,$\ \ $\xi=3S-\psi,$\ \ \ $\rho_{1}=\rho
,$\ \ \ $\rho_{2}=I_{\rho},$\ \ \ $k=G,$\ \ \ $b=D,$\ \ \ $3\gamma=4\delta
_{0},$\ \ \ $3\beta=4\gamma_{0},$ then the system (\ref{eq1}) becomes
\begin{equation}%
\begin{cases}
\rho_{1}w_{tt}-k\left(  w_{x}+\xi+s\right)  _{x}=0, & \text{ \ in\ }%
(0,L)\times\mathbb{R}^{+}\text{,}\\
\rho_{2}\xi_{tt}-b\xi_{xx}+k\left(  w_{x}+\xi+s\right)  =0, &\text{
\ in\ }(0,L)\times\mathbb{R}^{+}\text{,}\\
\rho_{2}s_{tt}-bs_{xx}+3k\left(  w_{x}+\xi+s\right)  +\gamma s+\beta
s_{t}=0, &\text{ \ in\ }(0,L)\times\mathbb{R}^{+}\text{.}%
\label{eq2}
\end{cases}
\end{equation}

Laminated beam models \eqref{eq2} have attracted significant interest due to the growing use of laminated materials in engineering applications, particularly in the automotive, aerospace, and construction industries, where their high strength, durability, and flexibility are highly valued (see, e.g., Mahajan \cite{Mahajan} and Reddy \cite{Reddy}). The model considered here, originally developed by Hansen and Spies \cite{Hansen2} from Timoshenko’s classical beam theory \cite{Timoshenko}, describes a two-layer beam joined by a thin adhesive layer of negligible mass and thickness. This layer contributes a restoring force proportional to the interfacial slip. While the first two equations in \eqref{eq2} stem from Timoshenko theory, the third captures the interlayer interaction and incorporates structural damping through an internal friction term.

%\subsection{Problem under consideration}
So, in context, we are interested in the control properties of the Laminated beam models \eqref{eq2} subject to a class of homogeneous dynamic Venttsel\footnote{The name of Alexander Ventcel is often spelled in various ways, such as Wentzell. We refer to the work \cite{Venttsel-a}, where this type of boundary condition was first introduced.} type boundary conditions and prescribed initial data, namely
\begin{equation}\label{2bbm}
\left\{\begin{array}{ll} 
\rho_{1}w_{tt}-k\left(  w_{x}+\xi+s\right)  _{x}   =0, &x\in(0,L),\,\,\, 
t>0\\   
\rho_{2}\xi_{tt}-b\xi_{xx}+k\left(  w_{x}+\xi+s\right)   
=0,  &x\in(0,L),\,\,\, 
t>0\\
\rho_{2}s_{tt}-bs_{xx}+3k\left(  w_{x}+\xi+s\right)  + \gamma s
=0, & x\in(0,L),\,\,\, t>0\\
w(0, t)=\xi(0, t)= s(0, t) =0, & t>0\\
w_{tt}(L, t)+w_x(L,t)+\xi(L,t)+s(L, t)=u_1(t), &
t>0\\
\xi_{tt}(L, t)+\xi_{x}(L, t)= u_2(t), &
t>0\\
s_{tt}(L, t) + s_x(L, t) =u_3(t), &
t>0\\
\left(  w, \xi, s \right)  \left( x,0 \right)  =\left(
w_{0}, \xi_{0}, s_{0}\right) ( x ), & x\in(0,L)\\
\left( w_{t}, \xi_{t}, s_{t}\right)  \left(  x,0\right)  =\left( w_{1}, \xi_{1}, s_{1}\right)  \left(  x\right), & x\in(0,L).\end{array}\right.\end{equation}

To be precise, this work aims to determine whether suitable boundary controls $u_1(t)$, $u_2(t)$, and $u_3(t)$ applied at the boundary of the beam can steer the system’s solutions to exhibit prescribed behaviors. This inquiry touches on a central problem in control theory:
\vglue 0.2 cm

\noindent\textbf{Null controllability problem:}  \textit{Given $T > 0$ and initial states $\left( w_{0}, w_{1}, \xi_{0}, \xi_{1}, s_{0}, s_{1}\right)$ within a specified function space, can we determine a   control inputs $u_1, u_2$ and $u_3$ such that the system \eqref{2bbm} admits a solution $\left(w, \xi, s\right)$ satisfying the initial conditions and the null control properties
$$\left( w, \xi, s \right) \left( x, T \right) = \left( 0, 0, 0 \right) \quad \text{and}\quad \left( w_{t}, \xi_{t}, s_{t} \right) \left( x, T \right) = \left( 0, 0, 0\right), \quad\text{for $x\in(0,L)$?}$$
}

\subsection{Background}Venttsel-type boundary conditions emerged independently in the works of Venttsel \cite{Venttsel-a} and Feller \cite{Feller}, originally in the study of generalized second-order differential operators and their connections to Markov processes. A broader framework was later developed in \cite{Venttsel}. In recent years, there has been growing interest in evolution equations incorporating Venttsel and dynamic boundary conditions, driven by their relevance in several applied contexts.

It is essential to note that these types of conditions occur in various mathematical models, particularly in fields such as population dynamics, heat transfer, continuum thermodynamics, and planar motion. For additional context and examples, the reader may consult references \cite{Farkas,Langer,Morgul-1,Morgul-2,Nikita,Tom} and the literature cited therein. A comprehensive treatment of the physical background and derivation of these boundary conditions can be found in the seminal work \cite{Goldstein}.

These boundary conditions have been the focus of significant research in the context of control and stabilization problems for partial differential equations. Several works address the stabilization of the wave equation under such boundary conditions. For instance, the exponential decay of energy has been established in domains of $\mathbb{R}^n$ and on compact Riemannian manifolds, as discussed in \cite{Buffe,Cavalcanti,Cavalcanti3,Hemina,Marta}, among others.

In the context of parabolic equations, the boundary null controllability of the heat equation with Wentzell-type boundary conditions has been studied in \cite{Maniar}. These conditions are particularly relevant in modeling physical systems where the boundary exhibits its dynamics, such as thermoelastic or fluid-structure interactions.

Boundary stabilization under dynamical boundary conditions has also been considered for various models.  In \cite{Buffe,Cavalcanti,Cavalcanti3}, the authors analyze the Cauchy--Ventcel problem, establishing uniform energy decay rates.  In \cite{Cavalcanti}, stabilization and observability estimates, along with asymptotic energy decay rates, are established for the wave equation with nonlinear interior damping and Wentzell-type boundary conditions. Also, we mention that in \cite{Cavalcanti3}, the authors investigate a wave equation subject to nonlinear interior damping and Wentzell-type boundary conditions of the form $\partial_\nu u + u = \Delta_T u$. They establish stabilization and observability estimates, as well as asymptotic energy decay rates for the system. A notable aspect of this work is that the damping mechanism is localized in a subset of the interior and does not extend to a full collar of the boundary, leaving a portion governed exclusively by the high-order Wentzell condition. This partial dissipation poses additional challenges in analyzing the long-time behavior of the solution.
More recently, in \cite{Buffe}, he studies the damped wave equation on either an open subset of $\mathbb{R}^n$ or a smooth Riemannian manifold with boundary, under Venttsel-type boundary conditions. By employing refined Carleman estimates near the boundary, they derive resolvent estimates along the imaginary axis, which yield logarithmic decay rates for the associated energy, uniform with respect to the thickness parameter $\delta$. 

Further investigations include the analysis of coupled systems with dynamical boundary interactions. For example, in \cite{CHITOUR}, the authors examine a one-dimensional Timoshenko beam system. In the first part of their work, they consider a single boundary fractional damping and prove strong stability without uniformity. They then establish polynomial energy decay rates depending on the system coefficients and the order of the damping. In the second part, the study focuses on exact boundary controllability under mixed Dirichlet–Neumann conditions. Using non-harmonic Fourier analysis and the Hilbert Uniqueness Method (HUM), they derive controllability results in appropriate functional spaces.

Another notable contribution is found in \cite{Victor}, where the authors investigate a Cauchy--Ventcel problem in an inhomogeneous medium, subject to nonlinear boundary damping distributed in a neighborhood $\omega$ of the boundary, satisfying the Geometric Control Condition. They obtain results on energy decay and stabilization in this nonlinear and spatially heterogeneous setting. 

Lastly, several works address the exponential stabilization of laminated beams with structural damping and boundary feedback controls. These studies consider mechanical models of layered composite materials and utilize boundary actuators to enforce rapid decay of vibrations, further illustrating the relevance of dynamical boundary conditions in practical control scenarios (see, e.g., \cite{Bau,FSou,LoT1,Wang} and the reference therein).

We emphasize that this represents only a small portion of the existing literature, and we encourage interested readers to consult the references cited in the works mentioned above.

\subsection{Main result and outline of the paper}  With this overview of the literature, we are in a position to present our main result. Before presenting the main result, let us introduce some notations. Consider the operator $\mathcal{A}$ associated to \eqref{2bbm} defined by
\begin{equation}\label{opae2-int}
\mathcal{A}U:=
\begin{bmatrix}
\Psi_1\\
\dfrac{1}{\rho_1}\left[k(w_x+\xi+s)_x\right]\\
\Psi_2\\
\dfrac{1}{\rho_2}\left[b \xi_{xx}-k(w_x+\xi+s)\right]\\
\Psi_3\\
\dfrac{1}{\rho_2}\left[b s_{xx} -3k(w_x+\xi+s)-\gamma s\right]\\
-(w_x(L)+\xi(L)+s(L))\\
-\xi_{x}(L)\\
-s_{x}(L)
\end{bmatrix}.
\end{equation}
Here, the new variables are
\[
\Psi_1=w_{t},\ \ \ \Psi_2=\xi_{t},\ \ \ \Psi_3=s_{t},\ \ \ \Psi_4(\cdot)=\Psi_1(L, \cdot), \ \ \ \Psi_5(\cdot)=\Psi_2(L, \cdot), \ \ \ \Psi_6(\cdot)=\Psi_3(L, \cdot),
\]  
and define the vector functions
\[
U=(w,\Psi_1, \xi,\Psi_2,s,\Psi_3, \Psi_4, \Psi_5, \Psi_6)^{\top}\text{ \ and \ }U_{0}=\left(  w_{0}%
,w_{1},\xi_{0},\xi_{1},s_{0},s_{1}, \Psi_4(0), \Psi_5(0), \Psi_6(0)\right)  ^{\top}.%
\]

Now, let us define the phase space $\mathcal{H}$, we consider the Lebesgue space $L^2 (0, L)$ with the usual inner product $\left\langle \cdot, \cdot \right\rangle$  and norm $\Vert \cdot \Vert$. We will also use the following Hilbert space:
$$
H_*^1(0, L)=\left\{w \in H^1(0, L) ; w(0)=0\right\},
$$
and define the phase space as
$$
\mathcal{H}=\left[H_*^1(0, L) \times L^2(0, L)\right]^3 \times\mathbb{R}^3.
$$
Endowed this space with the inner product
$$
\begin{aligned}
\left\langle U, U^{\sharp}\right\rangle_\mathcal{H}= & 3 \rho_1\left\langle \Psi_1, \Psi_1^{\sharp}\right\rangle+3 \rho_2\left\langle \Psi_2, \Psi_2^{\sharp}\right\rangle+\rho_2\left\langle \Psi_3, \Psi_3^{\sharp}\right\rangle+3 b\left\langle\xi_x, \xi_x^{\sharp}\right\rangle \\
& +b\left\langle s_x, s_x^{\sharp}\right\rangle + \gamma \left\langle s, s^{\sharp}\right\rangle +3 k\left\langle w_x+\xi+s, w_x^{\sharp}+\xi^{\sharp}+s^{\sharp}\right\rangle \\
&+3k\left(\Psi_4,\Psi_4^{\sharp}   \right)_{\mathbb{R}}+3b\left(\Psi_5,\Psi_5^{\sharp}   \right)_{\mathbb{R}}+b\left(\Psi_6,\Psi_6^{\sharp}   \right)_{\mathbb{R}}
\end{aligned}
$$
and the following norm
\begin{align*}
\|U\|^2_\mathcal{H}=&3 \rho_1\|\Psi_1\|^2+3 \rho_2\|\Psi_2\|^2+\rho_2\|\Psi_3\|^2+3 b\left\|\xi_x\right\|^2+b\left\|s_x\right\|^2 + \gamma \left\|s\right\|^2 +3 k\left\|w_x+\xi+s\right\|^2\\
&+3k \left| \Psi_4 \right|^2_{\mathbb{R}}+3b\left| \Psi_5 \right|^2_{\mathbb{R}}+b\left| \Psi_6 \right|^2_{\mathbb{R}},
\end{align*}
for $U=(w,\Psi_1, \xi,\Psi_2,s,\Psi_3, \Psi_4, \Psi_5, \Psi_6)^{\top}, \, \, U^{\sharp}=(w^{\sharp},\Psi_1^{\sharp}, \xi^{\sharp},\Psi_2^{\sharp}, s^{\sharp}, \Psi_3^{\sharp}, \Psi_4^{\sharp}, \Psi_5^{\sharp}, \Psi_6^{\sharp})^{\top} \in \mathcal{H}$. So, the operator $\mathcal{A}: D\left(\mathcal{A}\right) \subset \mathcal{H} \rightarrow \mathcal{H}$ defined by \eqref{opae2-int} has domain given by 
\begin{equation*}
D(\mathcal{A})=
\left\lbrace
	\begin{aligned}
		U=&(w,\Psi_1, \xi,\Psi_2,s,\Psi_3, \Psi_4, \Psi_5, \Psi_6)^{\top}\in \mathcal{H}; \,  (\omega, \xi, s)\in \left[ H^2(0,L)\cap H_*^1(0,L) \right]^3,  \\
	& (\Psi_1, \Psi_2, \Psi_3)\in \left[  H_*^1(0,L) \right]^3, \Psi_4=\Psi_1(L), \ \Psi_5=\Psi_2(L),  \ \Psi_6=\Psi_3(L),\\
    & (w_x(L), \xi_x(L), s_x(L)) \in \mathbb{R}^3.
	\end{aligned}
\right\rbrace
\end{equation*}

With this in hand, the following result ensures the controllability of the system \eqref{2bbm} considering controls acting in the Venttsel-type boundary conditions.
\begin{theorem}\label{controlteorfinalnodemos}Let a time $ T>0 $ and  given  $$\left(  w_{0}%
,w_{1},\xi_{0},\xi_{1},s_{0},s_{1}, ,w_{1}(L), \xi_{1}(L), s_{1}(L)\right)  ^{\top}\in D\left(\mathcal{A}\right),  $$ one can always find a control inputs  $ u_i \in  L^2(0,T), \ \ i=1,2,3  $,  such that \eqref{2bbm} admits a unique solution $$ U=(w,w_t, \xi,\xi_t,s, s_t, w_t(L), \xi_t(L), s_t(L))^{\top} \in
C ([0,T]; D\left(\mathcal{A}\right))$$satisfying 
$$ (w(T), w_t(T),  \xi(T), \xi_t(T),  s(T),  s_t(T), w_t(L, T), \xi_t(L, T), s_t(L, T))=(0, 0, 0, 0, 0, 0, 0, 0, 0)$$
Moreover, there exists a constant $ C>0 $ such that
\begin{align*}
\Vert (u_1, u_2, u_3) \Vert_ {[L^2(0,T)]^3 }\leq C     \Vert \left(  w_{0}%
,w_{1},\xi_{0},\xi_{1},s_{0},s_{1}, ,w_{1}(L), \xi_{1}(L), s_{1}(L)\right)  ^{\top} \Vert_{\mathcal{H}}. 
\end{align*}
\end{theorem}

To establish Theorem \ref{controlteorfinalnodemos}, we employ the observation–control duality framework developed by Dolecki and Russell \cite{DoRu}, within the setting introduced by Lions \cite{Lions1988SIAM}. Specifically, we rely on the Hilbert Uniqueness Method (HUM), which provides a fundamental equivalence between controllability and observability. A key contribution of our work lies in addressing a Timoshenko-type system involving three coupled equations. This significantly increases the complexity of the analysis, as the model consists of three beams subject to dynamical boundary conditions. These boundary conditions introduce notable difficulties in establishing the observability inequality. To overcome this, we derive two intermediate inequalities—see \eqref{primerdesiobserder} and \eqref{seconddesiobserder2} in Section \ref{sec:3}—which play a crucial role in handling the full system.

\vspace{0.2cm}

We conclude the introduction with an outline of the paper. In Section \ref{sec:2}, we prove the well-posedness of system \eqref{2bbm} using semigroup theory. Section \ref{sec:3} is devoted to the proof of our main result. We begin by establishing an observability inequality for the adjoint system associated with \eqref{2bbm}, which leads to the proof of Theorem \ref{controlteorfinalnodemos}. Finally, in Section \ref{sec:4}, we present concluding remarks and discuss directions for future research.

\section{Well-posedness theory}\label{sec:2}

This section is dedicated to establishing the well-posedness of system \eqref{2bbm}. We begin by analyzing the homogeneous case and then extend the results to the nonhomogeneous system. These foundational results are essential for addressing the controllability problem outlined in the introduction.

\subsection{The homogeneous system}

Let us first consider the following homogeneous system
\begin{equation}\label{2bbm-hom}
\left\{\begin{array}{ll} \rho_{1}w_{tt}-k\left(  w_{x}+\xi+s\right)_{x}   =0, &x\in(0,L),\,\,\, 
t>0,\\   
\rho_{2}\xi_{tt}-b\xi_{xx}+k\left(  w_{x}+\xi+s\right)   
=0,  &x\in(0,L),\,\,\, 
t>0,\\
\rho_{2}s_{tt}-bs_{xx}+3k\left(  w_{x}+\xi+s\right) + 
\gamma s=0, & x\in(0,L),\,\,\, t>0,\\
w(0, t)=\xi(0, t)= s(0, t) =0, & t>0,\\
w_{tt}(L, t)+w_x(L,t)+\xi(L,t)+s(L, t)=0, &
t>0,\\
\xi_{tt}(L, t)+\xi_{x}(L, t)= 0, &
t>0,\\
s_{tt}(L, t) + s_x(L, t) =0, &
t>0,\\
\left(  w, \xi, s \right)  \left( x,0 \right)  =\left(
w_{0}, \xi_{0}, s_{0}\right) ( x ), & x\in(0,L),\\
\left( w_{t}, \xi_{t}, s_{t}\right)  \left(  x,0\right)  =\left( w_{1}, \xi_{1}, s_{1}\right)  \left(  x\right), & x\in(0,L).\end{array}\right.
\end{equation}
Remember that
\[
\Psi_1=w_{t},\ \ \ \Psi_2=\xi_{t},\ \ \ \Psi_3=s_{t},\ \ \ \Psi_4(\cdot)=\Psi_1(L, \cdot), \ \ \ \Psi_5(\cdot)=\Psi_2(L, \cdot), \ \ \ \Psi_6(\cdot)=\Psi_3(L, \cdot),
\]  
and
\[
U=(w,\Psi_1, \xi,\Psi_2,s,\Psi_3, \Psi_4, \Psi_5, \Psi_6)^{\top}\text{ \ and \ }U_{0}=\left(  w_{0}%
,w_{1},\xi_{0},\xi_{1},s_{0},s_{1}, \Psi_4(0), \Psi_5(0), \Psi_6(0)\right)  ^{\top}.%
\]
With  these notations in hand, the system \eqref{2bbm-hom} can be written as an abstract evolution equation
\begin{equation}\label{abs}
\left\{
\begin{array}
[c]{l}
U_t = \mathcal{A} U,\\
U(0) = U_0,
\end{array}\right.
\end{equation}where the operator $\mathcal{A}$ is defined by

%\begin{equation}\label{opae2}
%\left(\begin{array}%{cccccc}0 & 1 & 0 & 0 & 0 & 0 
%\\ 
%\frac{k}{\rho_1} \partial_{x x} & 0 & \frac{k}{\rho_1} \partial_{x} & 0 & \frac{k}{\rho_1} \partial_x & 0
%\\ 
%0 & 0 & 0 & 1 & 0 & 0 
%\\
%\frac{k}{\rho_1} \partial_{x} & 0 & \frac{1}{\rho_2}\left(b \partial_{x x}-k\right) & 0 & \frac{-k}{\rho_2} & 0 
%\\
%0 & 0 & 0 & 0 & 0 & 1 
%\\
%\frac{-3 k}{\rho_2} \partial_x & 0 & -3\frac{ k}{\rho_2} & 0 & \frac{1}{\rho_2}\left(\partial_{xx}-3k \right) & 0\end{array}\right).
%\end{equation}

\begin{equation}\label{opae2}
\mathcal{A}U:=
\begin{bmatrix}
\Psi_1\\
\dfrac{1}{\rho_1}\left[k(w_x+\xi+s)_x\right]\\
\Psi_2\\
\dfrac{1}{\rho_2}\left[b \xi_{xx}-k(w_x+\xi+s)\right]\\
\Psi_3\\
\dfrac{1}{\rho_2}\left[b s_{xx} -3k(w_x+\xi+s)-\gamma s\right]\\
-(w_x(L)+\xi(L)+s(L))\\
-\xi_{x}(L)\\
-s_{x}(L)
\end{bmatrix}.
\end{equation}
Taking into account the spaces defined in the introduction and denoting by  $\rho(\mathcal{A})$ the resolvent set of the operator $\mathcal{A}$. The following result ensures the invertibility of the operator $\mathcal{A}$.
\begin{lemma}\label{0inresolv}
  Let $\mathcal{H}$ and $\left(\mathcal{A},D\left(\mathcal{A}\right)\right)$ be defined as before. Then, $0 \in \rho(\mathcal{A})$. Moreover, $\mathcal{A}^{-1}$ is compact.
\end{lemma}
\begin{proof}
For $F=\left(f_1, f_2, f_3, f_4, f_5, f_6, f_7, f_8, f_9 \right) \in \mathcal{H}$, we will show the existence of $$U=(w,\Psi_1, \xi,\Psi_2,s,\Psi_3, \Psi_4, \Psi_5, \Psi_6)^{\top} \in D\left(\mathcal{A}\right),$$ unique solution of the equation
$$
\mathcal{A} U=F.
$$
Equivalently, one must consider the system given by    
\begin{equation}
\begin{cases}
 \Psi_1= f_1, \quad \Psi_2=f_3, \quad \Psi_3=f_5, \label{eq:system1_1}\\
 k\left(w_x+\xi+s\right)_x=\rho_1 f_2, \\
 b \xi_{xx}-k(w_x+\xi+s)=\rho_2 f_4, \\
 b s_{xx}-3k(w_x+\xi+s) -\gamma s=\rho_2f_6,\\
 -(w_x(L)+\xi(L)+s(L)) = f_7,\\
-\xi_x(L)=f_8, \\
-s_x(L)=f_9,
\end{cases}
\end{equation}
with the following boundary conditions
\begin{equation}\label{eq:system_cond1}
w(0)=\xi(0)=s(0)=0.
\end{equation}

Consider $(w^{*}, \xi^{*},  s^{*}) \in  [H_*^1(0,L)]^3$, so multiplying the equation $\eqref{eq:system1_1}_2$, $\eqref{eq:system1_1}_3$,  $\eqref{eq:system1_1}_4$, $\eqref{eq:system1_1}_5$,  $\eqref{eq:system1_1}_6$ and $\eqref{eq:system1_1}_7$ by  $3w^{*},$ $ 3\xi^{*},$ $ s^{*}$, $3kw^{*}(L),$ $ 3b\xi^{*}(L),$ $ bs^{*}(L)$,  respectively, integrating by parts in $(0,L)$ and taking the sum, we get that
\begin{equation} \label{eq:bilinear_def0}
\mathcal{B}((w, \xi,  s),(w^{*}, \xi^{*},  s^{*}))=\mathcal{L}(w^{*}, \xi^{*},  s^{*}), \quad \forall(w^{*}, \xi^{*},  s^{*}) \in  [H_*^1(0,L)]^3,
\end{equation}
thanks to the boundary conditions \eqref{eq:system_cond1}. Here
\begin{equation}
\begin{aligned}\label{eq:bilinear_def1}
\mathcal{B}((w, \xi,  s),(w^{*}, \xi^{*},  s^{*})) =  3b\left\langle\xi_x, \xi_x^{*}\right\rangle + b\left\langle s_x, s_x^{*}\right\rangle  +  \gamma \left\langle s, s^{*}\right\rangle +  3k\left\langle w_x+\xi+s, w_x^{*}+\xi^{*}+s^{*}\right\rangle,
\end{aligned}
\end{equation}and 
\begin{equation}\label{eq:bilinear_def2}
\begin{split}
  \mathcal{L}(w^{*}, \xi^{*}, s^{*})=&-3\rho_1\left\langle f_1, w^{*}\right\rangle-3\rho_2\left\langle f_4, \xi^{*}\right\rangle-\rho_2\left\langle f_6, s^{*}\right\rangle\\
  &+3kf_7w^{*}(L) + 3bf_8 \xi^{*}(L)+ bf_9 s^{*}(L).
  \end{split}
\end{equation}

Now, observe that $$\mathcal{B}:\left[H_*^1(0,L)\right]^3 \times [H_*^1(0,L)]^3 \longrightarrow \mathbb{C}$$ defined by \eqref{eq:bilinear_def1} is a bilinear form which one is continuous and coercive. Moreorver, $$\mathcal{L}:   [H_*^1(0,L)]^3\longrightarrow \mathbb{C},$$ given by \eqref{eq:bilinear_def2}, is a linear continuous form. Then, using Lax-Milgram theorem, we deduce that there exists $(w, \xi,  s) \in[H_*^1(0,L)]^3$ a unique weak solution of the variational Problem  \eqref{eq:bilinear_def0}. From $\eqref{eq:system1_1}_1$ and applying the classical elliptic regularity theory we deduce that $$U=(w,\Psi_1, \xi,\Psi_2,s,\Psi_3, \Psi_4, \Psi_5, \Psi_6)^{\top} \in D\left(\mathcal{A}\right)$$  and, consequently, $\mathcal{A}$ is bijective.

Following similar calculations as done in \cite[Proposition 2]{CQS}, we prove that $\mathcal{A}^{-1}$ is bounded. Thus, $0 \in \rho(\mathcal{A})$. Finally, by Sobolev’s embedding theorem, we deduce that the canonical injection $$i: D\left(\mathcal{A}\right) \hookrightarrow \mathcal{H}$$ is compact, and therefore, $\mathcal{A}^{-1}$ is also compact, showing the proposition. \end{proof}

The next result ensures that $\mathcal{A}$ generates a group. The result can be read as follows. 
\begin{theorem}\label{teogeneinfiA}
The operator $\left(\mathcal{A}, D\left(\mathcal{A}\right)\right)$  is an infinitesimal generator of a group of isometries $\{S(t)\}_{t\in\mathbb{R}}$ in $\mathcal{H}$.
\end{theorem}

\begin{proof}
Let $$U^{\sharp}=(w^{\sharp},\Psi_1^{\sharp}, \xi^{\sharp},\Psi_2^{\sharp}, s^{\sharp}, \Psi_3^{\sharp}, \Psi_4^{\sharp}, \Psi_5^{\sharp}, \Psi_6^{\sharp})^{\top} \in D\left(\mathcal{A}\right).$$  Then, for all $$U=(w,\Psi_1, \xi,\Psi_2,s,\Psi_3, \Psi_4, \Psi_5, \Psi_6)^{\top} \in D\left(\mathcal{A}\right),$$ we deduce that
\begin{align*}
\left\langle 
\mathcal{A}U, U^{\sharp}\right\rangle_\mathcal{H} = & 3 \rho_1\left\langle \dfrac{1}{\rho_1}\left[k(w_x+\xi+s)_x\right], \Psi_1^{\sharp}\right\rangle + 3 \rho_2\left\langle \dfrac{1}{\rho_2}\left[b \xi_{xx}-k(w_x+\xi+s)\right], \Psi_2^{\sharp}\right\rangle \\
&+ \rho_2\left\langle \dfrac{1}{\rho_2}\left[b s_{xx} -3k(w_x+\xi+s) -\gamma s\right], \Psi_3^{\sharp}\right\rangle +3 b\left\langle \Psi_{2,x}, \xi_x^{\sharp}\right\rangle  +b\left\langle \Psi_{3,x}, s_x^{\sharp}\right\rangle   + \gamma \left\langle \Psi_{3}, s^{\sharp}\right\rangle \\
&+  3 k\left\langle \Psi_{1,x}+\Psi_2+\Psi_3, w_x^{\sharp}+\xi^{\sharp}+s^{\sharp}\right\rangle   + 3k\left(-(w_x(L)+\xi(L)+s(L)) ,\Psi_4^{\sharp}   \right)_{\mathbb{R}}\\
&+3b\left(-\xi_x(L),\Psi_5^{\sharp}   \right)_{\mathbb{R}}+b\left(-s_x(L),\Psi_6^{\sharp}   \right)_{\mathbb{R}}\\ 
= & 3 \rho_1\left\langle \Psi_1, -\dfrac{1}{\rho_1}\left[k(w^{\sharp}_x+\xi^{\sharp}+s^{\sharp})_x\right]\right\rangle+3 \rho_2\left\langle \Psi_2, - \dfrac{1}{\rho_2}\left[b \xi^{\sharp}_{xx}-k(w^{\sharp}_x+\xi^{\sharp}+s^{\sharp})\right]\right\rangle\\
&+\rho_2\left\langle \Psi_3, -\dfrac{1}{\rho_2}\left[b s^{\sharp}_{xx} -3k(w^{\sharp}_x+\xi^{\sharp}+s^{\sharp})-\gamma s^{\sharp}\right]\right\rangle+3 b\left\langle\xi_x, -\Psi_{2,x}^{\sharp}\right\rangle +b\left\langle s_x, -\Psi_{3,x}^{\sharp}\right\rangle \\
& + \gamma \left\langle s, -\Psi_{3}^{\sharp}\right\rangle +3 k\left\langle w_x+\xi+s, -(\Psi_{1,x}^{\sharp}+\Psi_2^{\sharp}+\Psi_3^{\sharp})\right\rangle + 3k\left(\Psi_4 ,(w^{\sharp}_x(L)+\xi^{\sharp}(L)+s^{\sharp}(L))   \right)_{\mathbb{R}}\\
&+3b\left(\Psi_5,\xi_x^{\sharp} (L)  \right)_{\mathbb{R}}+b\left(\Psi_6,s_x^{\sharp}(L)   \right)_{\mathbb{R}}\\
=&\left\langle 
U, -\mathcal{A} U^{\sharp}\right\rangle_\mathcal{H}.
\end{align*}Hence, $U^{\sharp} \in D\left(\mathcal{A}^{*}\right)$ and $\mathcal{A}^{*}U^{\sharp} = -\mathcal{A}U^{\sharp}$, for any $U^{\sharp} \in D\left(\mathcal{A}\right)$. 

On the other hand, let $U^{\sharp} =(w^{\sharp},\Psi_1^{\sharp}, \xi^{\sharp},\Psi_2^{\sharp}, s^{\sharp}, \Psi_3^{\sharp}, \Psi_4^{\sharp}, \Psi_5^{\sharp}, \Psi_6^{\sharp})^{\top}\in D\left(\mathcal{A}^{*}\right)$. Then, there exists $U_0 \in \mathcal{H}$, such that 
\begin{align}\label{adj_1}
\left\langle 
\mathcal{A}U, U^{\sharp}\right\rangle_\mathcal{H} = \left\langle 
U, U_0\right\rangle_\mathcal{H}, \text{ for all }\ \ U=(w,\Psi_1, \xi,\Psi_2,s,\Psi_3, \Psi_4, \Psi_5, \Psi_6)^{\top} \in D\left(\mathcal{A}\right).
\end{align}
By the same calculations as done before, we have that
\begin{equation}\label{adj_2}
\begin{split}
\left\langle 
\mathcal{A}U, U^{\sharp}\right\rangle_\mathcal{H}  =   &\left\langle 
U, -\mathcal{A} U^{\sharp}\right\rangle_\mathcal{H} +  3k\left((w_x(L)+\xi(L)+s(L)) , \Psi_1^{\sharp}(L)-\Psi_4^{\sharp}   \right)_{\mathbb{R}} \\
&+3b\left(\xi_x(L), \Psi_2^{\sharp}(L)-\Psi_5^{\sharp}   \right)_{\mathbb{R}}+b\left(s_x(L),\Psi_3^{\sharp}(L)-\Psi_6^{\sharp}   \right)_{\mathbb{R}}  \\
&-3k\left(w_x(0) , \Psi_1^{\sharp}(0)   \right)_{\mathbb{R}} -3b\left(\xi_x(0), \Psi_2^{\sharp}(0)   \right)_{\mathbb{R}}-b\left(s_x(0),\Psi_3^{\sharp}(0)   \right)_{\mathbb{R}}. 
\end{split}
\end{equation}Therefore, taking the diference between \eqref{adj_1} and \eqref{adj_2} we obtain the following identity
\begin{align*}
0=&\left\langle U, U_0+\mathcal{A} U^{\sharp}\right\rangle_\mathcal{H} + 3k\left((w_x(L)+\xi(L)+s(L)) , \Psi_4^{\sharp} - \Psi_1^{\sharp}(L)  \right)_{\mathbb{R}} \\
&+3b\left(\xi_x(L), \Psi_5^{\sharp} - \Psi_2^{\sharp}(L)  \right)_{\mathbb{R}}  +b\left(s_x(L),\Psi_6^{\sharp} - \Psi_3^{\sharp}(L)  \right)_{\mathbb{R}} + 3k\left(w_x(0) , \Psi_1^{\sharp}(0)   \right)_{\mathbb{R}} \\
&+3b\left(\xi_x(0), \Psi_2^{\sharp}(0)   \right)_{\mathbb{R}}+b\left(s_x(0),\Psi_3^{\sharp}(0)   \right)_{\mathbb{R}}, \end{align*}
for all $U\in D\left(\mathcal{A}\right).$ From the  above equality, we obtain that $U^{\sharp} \in D\left(\mathcal{A}\right)$ and 
\begin{align*}
 \mathcal{A}^{*}U^{\sharp} = U_0=-\mathcal{A}U^{\sharp}, \ \  \text{ for any}  \ U^{\sharp} \in D\left(\mathcal{A}^{*}\right). 
\end{align*}Thus, $D\left(\mathcal{A}\right)=D\left(\mathcal{A}^{*}\right)$ and $\mathcal{A}^{*}= - \mathcal{A}$. We deduce that   $\mathcal{A}$ is a skew-adjoint operator  and,
therefore, generates a group of isometries $\{S(t)\}_{t\in\mathbb{R}}$ in $\mathcal{H}$, giving the result.    
\end{proof}

As a direct consequence of Lemma \ref{0inresolv} and Theorem \ref{teogeneinfiA}, and by applying semigroup theory for evolution equations\footnote{See, for instance, \cite{Pazy}.}, we obtain the following existence and uniqueness result.
\begin{theorem}\label{ex-hom}
For any $U_0 \in D\left(\mathcal{A}\right)$, there exists a unique solution $U$ of the system \eqref{abs} such that
$$
U  \in C\left([0, \infty); D\left(\mathcal{A}\right) \right) \cap C^1\left([0, \infty); \mathcal{H}\right).
$$
Moreover, for any $t>0$, we have that
\begin{align}\label{identi_gisometr}
\Vert U(t) \Vert^2_{\mathcal{H}} =\Vert U_0 \Vert^2_{\mathcal{H}}.
\end{align}
\end{theorem}

\subsection{The nonhomogeneous system}

In this subsection, we focus on the complete system \eqref{2bbm} and begin by presenting the following well-posedness result.

 \begin{theorem}%\label{ex-full}
For any $U_0 \in D\left(\mathcal{A}\right)  $
and $ u_i \in  C^{\infty}_0(0,\infty)$, for $ i=1,2,3 $, system \eqref{2bbm} has a unique solution  $U  \in C\left([0, \infty); D\left(\mathcal{A}\right) \right) \cap C^1\left([0, \infty); \mathcal{H}\right).$
\end{theorem}

\begin{proof}
Let $ \phi_1 \in C^{\infty}\left[ 0, L \right]  $ be a cut-off function, such that $ \phi_1(0)=\phi_1(L)=0 $ and $ \phi_1^{\prime}(L)=-1 $. First, consider the change of variables
\begin{equation}\label{change}
\left(\begin{array}{c}
z \\ \eta \\ \varphi \end{array}\right) = \left(\begin{array}{c}
w \\ \xi \\ s \end{array}\right) - \left(\begin{array}{c}
\tilde{w} \\ \tilde{\xi} \\ \tilde{s} \end{array}\right) + \left(\begin{array}{c}
u_1(t)\phi_1(x) \\ u_2(t)\phi_1(x) \\ u_3(t)\phi_1(x) \end{array}\right).
\end{equation}
Here $ \left(\tilde{w}, \tilde{\xi}, \tilde{s}\right)$ is the unique  solution of the system
\begin{equation*}
\left\{\begin{array}{ll} \rho_{1}\tilde{w}_{tt}-k(  \tilde{w}_{x}+\tilde{\xi}+\tilde{s})_{x}   =0, &x\in(0,L),\,\,\, 
t>0,\\   
\rho_{2}\tilde{\xi}_{tt}-b\tilde{\xi}_{xx}+k( \tilde{ w}_{x}+\tilde{\xi}+\tilde{s})   
=0,  &x\in(0,L),\,\,\, 
t>0,\\
\rho_{2}\tilde{s}_{tt}-b\tilde{s}_{xx}+3k(  \tilde{w}_{x}+\tilde{\xi}+\tilde{s})  
+\gamma \tilde{s}=0, & x\in(0,L),\,\,\, t>0,\\
\tilde{w}(0, t)=\tilde{\xi}(0, t)= \tilde{s}(0, t) =0, & t>0,\\
\tilde{w}_{tt}(L, t)+\tilde{w}_x(L,t)+\tilde{\xi}(L,t)+\tilde{s}(L, t)=0, &
t>0,\\
\tilde{\xi}_{tt}(L, t)+\tilde{\xi}_{x}(L, t)= 0, &
t>0,\\
\tilde{s}_{tt}(L, t) + \tilde{s}_x(L, t) =0, &
t>0,\\
\left(  \tilde{w},\tilde{ \xi}, \tilde{s} \right)  \left( x,0 \right)  =\left(
w_{0}, \xi_{0}, s_{0}\right) ( x ), & x\in(0,L),\\
\left( \tilde{w}_{t}, \tilde{\xi}_{t}, \tilde{s}_{t}\right)  \left(  x,0\right)  =\left( w_{1}, \xi_{1}, s_{1}\right)  \left(  x\right), & x\in(0,L),\end{array}\right.
\end{equation*}
given by Theorem \ref{ex-hom}. Thus, $ \left(z, \eta, \varphi\right)$ solves the following problem
\begin{equation}\label{siste23}
\left\{\begin{array}{ll} \rho_{1}z_{tt}-k\left(  z_{x}+\eta+\varphi\right)  _{x}   =F, &x\in(0,L),\,\,\, 
t>0,\\   
\rho_{2}\eta_{tt}-b\eta_{xx}+k\left(  z_{x}+\eta+\varphi\right)   
=G,  &x\in(0,L),\,\,\, 
t>0,\\
\rho_{2}\varphi_{tt}-b\varphi_{xx}+3k\left(  z_{x}+\eta+\varphi\right)  
+ \gamma \varphi=H, & x\in(0,L),\,\,\, t>0,\\
z(0, t)=\eta(0, t)= \varphi(0, t) =0, & t>0,\\
z_{tt}(L, t)+z_x(L,t)+\eta(L,t)+\varphi(L, t)=0, &
t>0,\\
\eta_{tt}(L, t)+\eta_{x}(L, t)= 0, &
t>0,\\
\varphi_{tt}(L, t) + \varphi_x(L, t) =0, &
t>0,\\
\left(  z, \eta, \varphi \right)  \left( x,0 \right)  =\left(0, 0, 0\right), & x\in(0,L),\\
\left( z_{t}, \eta_{t}, \varphi_{t}\right)  \left(  x,0\right)  =\left( 0, 0, 0\right), & x\in(0,L),\end{array}\right.
\end{equation}where
\begin{equation*}
 \left\{\begin{array}{ll}  
F(t, x)= &\rho_{1} u_1^{(2)}(t)\phi_1(x) 
-k \left[ u_1(t)\phi_1^{(2)}(x) + \left ( u_2(t) + u_3(t)         \right)\phi^{\prime}_1(x) \right],\\
\\
G(t, x)=& \rho_{2} u_2^{(2)}(t)\phi_2(x) -b u_2(t) \phi_2^{(2)}(x) + k \left[ u_1(t)\phi_1^{\prime}(x) + \left ( u_2(t) + u_3(t)         \right)\phi_1(x) \right]\\
\\
H(t, x)=&\rho_{2} u_3^{(2)}(t)\phi_3(x) -b u_3(t) \phi_3^{(2)}(x) + 3k \left[ u_1(t)\phi_1^{\prime}(x) + \left ( u_2(t) + u_3(t)         \right)\phi_1(x) \right] + \gamma u_3(t)\phi_1(x),
 \end{array}\right.
\end{equation*}and
$$ \left(\begin{array}{c}
F(t,x) \\ G(t,x) \\ H(t, x)\end{array}\right)\in \left[
C^1(\left[ 0,\infty \right]\times \left[ 0, L \right] )\right]^3.$$

Observe that, by employing the same strategy as in the previous section, system \eqref{siste23} can be reformulated as the following abstract evolution equation:
\begin{equation}\label{sistem24}
\left\{
\begin{array}
[c]{l}
W_t + \mathcal{A} W= \mathcal{K}\\
W(0) = 0,
\end{array}\right.\nonumber
\end{equation}
where
$$
 W = (z, z_t, \eta, \eta_t, \varphi, 
 \varphi_t, z_t(L, T), \eta_t(L, t), \varphi_t(L, t) )^{\top},
 $$
 and 
 $$
 \mathcal{K} = \left(0, \frac{F}{\rho_1}, 0, \frac{G}{\rho_2}, 0, \frac{H}{\rho_2}, 0, 0, 0 \right)^{\top} \in\left[
C^1(\left[ 0,\infty \right]\times \left[ 0, L \right] )\right]^9.
$$
Finally, since $\mathcal{A}$ generates a group of isometries in $\mathcal{H}$, it follows that system \eqref{siste23} admits a unique solution
$$W \in C\left([0, \infty); D(\mathcal{A})\right) \cap C^1\left([0, \infty); \mathcal{H}\right).$$
Reverting the change of variables in \eqref{change}, we thus complete the proof.
\end{proof}

Using the previous well-posedness results for the full system, we will study solutions of the system \eqref{2bbm}  in the transposition sense. 
\begin{definition}
Given $T > 0$, $U_0 \in \mathcal{H}'$, $(h_1, h_2, h_3) \in  \left(L^2\left(0,T;  \left(H_*^1(0,L) \right)^{*}\right)\right)^3 $
and $ u_i \in  L^2(0,T),$ for $i=1,2,3 $, let us consider the non-homogeneous system given by
\begin{equation}\label{2bbm_nonhomo_2}
\left\{\begin{array}{ll} 
\rho_{1}w_{tt}-k\left(  w_{x}+\xi+s\right)_{x}   =h_1, &x\in(0,L),\,\,\, 
t \in (0, T),\\   
\rho_{2}\xi_{tt}-b\xi_{xx}+k\left(  w_{x}+\xi+s\right)   
=h_2,  &x\in(0,L),\,\,\, 
t\in (0, T),\\
\rho_{2}s_{tt}-bs_{xx}+3k\left(  w_{x}+\xi+s\right)  
+ \gamma s=h_3, & x\in(0,L),\,\,\, t\in (0, T),\\
w(0, t)=\xi(0, t)= s(0, t) =0, & t\in (0, T),\\
w_{tt}(L, t)+w_x(L,t)+\xi(L,t)+s(L, t)=u_1(t), &
t\in (0, T),\\
\xi_{tt}(L, t)+\xi_{x}(L, t)= u_2(t), &
t\in (0, T),\\
s_{tt}(L, t) + s_x(L, t) =u_3(t), &
t\in (0, T),\\
\left(  w, \xi, s \right)  \left( x,0 \right)  =\left(
w_{0}, \xi_{0}, s_{0}\right) ( x ), & x\in(0,L)\\
\left( w_{t}, \xi_{t}, s_{t}\right)  \left(  x,0\right)  =\left( w_{1}, \xi_{1}, s_{1}\right)  \left(  x\right), & x\in(0,L).\end{array}\right.\end{equation}
A vector function  $$U=(w,w_t, \xi,\xi_t,s, s_t, w_t(L), \xi_t(L), s_t(L))^{\top}\in C ([0,T]; D\left(\mathcal{A}\right)),$$ is said to be a solution by transposition to \eqref{2bbm_nonhomo_2} if the following identity holds
\begin{align}\label{equa2}
& \displaystyle \left\langle  \left(\begin{array}{c}
3 \rho_1 w_t(\cdot, \tau) \\ -3 \rho_1 w(\cdot, \tau)  \\ 3 \rho_2 \xi_t(\cdot, \tau) \\-3 \rho_2 \xi(\cdot, \tau)  \\
 \rho_2 s_t(\cdot, \tau)  \\ -\rho_2 s(\cdot, \tau) \\
 -3kw(L, \tau)\\
 -3b\xi(L, \tau)\\
 -bs(L, \tau)\end{array}\right), \left(\begin{array}{c}
 \chi_0^{\tau}\\ \chi_1^{\tau} \\ \eta_0^{\tau} \\ \eta_1^{\tau}\\ \Theta_0^{\tau} \\ \Theta_1^{\tau} \\\chi_t(L, \tau)\\ \eta_t(L, \tau)\\ \Theta_t(L, \tau) \end{array}\right)  \displaystyle \right\rangle=\displaystyle \left\langle  \left(\begin{array}{c}
3 \rho_1 w_1 \\ -3 \rho_1 w_0 \\ 3 \rho_2 \xi_1\\-3 \rho_2 \xi_0 \\
 \rho_2 s_1 \\ -\rho_2 s_0\\
 -3kw(L, 0)\\
 -3b\xi(L, 0)\\
 -bs(L, 0)\end{array}\right), \left(\begin{array}{c}
 \chi(0)\\ \chi_t(0) \\ \eta(0) \\ \eta_t(0)\\ \Theta(0) \\ \Theta_t(0) \\\chi_t(L, 0)\\ \eta_t(L, 0)\\ \Theta_t(L, 0) \end{array}\right)  \displaystyle \right\rangle \\
& +\displaystyle\int_{0}^{\tau} \, \left \langle \left(3h_1(t), 3h_2(t, h_3(t)) \right), \left(\chi(t), \eta(t), \Theta(t) \right)\right\rangle_{\left[\left(  H_*^1(0,L) \right)^{*}, H_*^1(0,L) \right]^3}dxdt \nonumber\\
&+\left((3ku_1(t), 3bu_2(t), bu_3(t)),  1_{(0, \tau)}(\chi(L, t), \eta(L, t), \Theta(L, t) )\right)_{[L^2(0, T)]^3} \nonumber \\
&+ \left( \left(3kw_t(L, 0), 3b\xi_t(L, 0), bs_t(L, 0), \left( \chi(L, 0), \eta(L, 0), \Theta(L, 0) \right)\right)\right)_{\mathbb{R}^3}\nonumber
\end{align}
for any $ \tau \in  [0, T]$ and $ W^{\tau } \in  \mathcal{H}$, with $$W^{\tau }=(\chi_0^{\tau}, \chi_1^{\tau}, \eta_0^{\tau}, \eta_1^{\tau}, \Theta_0^{\tau}, \Theta_1^{\tau}, \chi_1^{\tau}(L), \eta_1^{\tau}(L), \Theta_1^{\tau}(L))^{\top},$$
where $\left\langle \cdot , \cdot\right\rangle$ is the
duality bracket in $\mathcal{H}'\times\mathcal{H}$ and  $W=\left( \chi, \chi_t, \eta, \eta_t, \Theta, \Theta_t,  \chi_t(L, t), \eta_t(L, t), \Theta_t(L, t) \right)^{\top} $ is solution of the following adjoint system
\begin{equation}\label{sistema3}
\left\{\begin{array}{ll} \rho_{1}\chi_{tt}-k\left(  \chi_{x}+\eta + \Theta\right)  _{x}   =0, &x\in(0,L),\,\,\, 
t\in (0, \tau),\\   
\rho_{2}\eta_{tt}-b\eta_{xx}+k\left(  \chi_{x}+\eta + \Theta\right)   
=0,  &x\in(0,L),\,\,\, 
t\in (0, \tau),\\
\rho_{2}\Theta_{tt}-b\Theta_{xx}+3k\left(  \chi_{x}+\eta + \Theta\right)  
+ \gamma \Theta=0, & x\in(0,L),\,\,\, t\in (0, \tau),\\
\chi(0, t)=\eta(0, t)= \Theta(0, t) =0, & t\in (0, \tau)\\
\chi_{tt}(L, t)+\chi_x(L,t)+\eta(L,t)+\Theta(L, t)=0, &
t\in (0, \tau),\\
\eta_{tt}(L, t)+\Theta_{x}(L, t)= 0, &
t\in (0, \tau),\\
\Theta_{tt}(L, t) + \Theta_x(L, t) =0, &
t\in (0, \tau),\\
\left(  \chi, \eta, \Theta \right)  \left( x,\tau \right)  =(\chi_0^{\tau},  \eta_0^{\tau},  \Theta_0^{\tau}), & x\in(0,L),\\
\left( \chi_{t}, \eta_{t}, \Theta_{t}\right)  \left(  x, \tau\right)  =(\chi_1^{\tau},  \eta_1^{\tau},  \Theta_1^{\tau}), & x\in(0,L), \\
\left(  \chi, \eta, \Theta \right)  \left( L,\tau \right)  =(0, 0, 0).
\end{array}\right.
\end{equation}
\end{definition}

The next theorem establishes the existence and uniqueness of solutions for the system \eqref{2bbm} in the transposition sense.

\begin{theorem}%\label{exist-transp}
Let $T > 0$, $U_0 \in D\left(\mathcal{A}\right)  $, $(h_1, h_2, h_3) \in  \left(L^2\left(0,T;  \left( H_*^1(0,L) \right)^{*}\right)\right)^3 $
and $ u_i \in  L^2(0,T),$ for  $ i=1,2,3 $. Then, there exists a unique solution 
$$ U=(w,w_t, \xi,\xi_t,s, s_t, w_t(L), \xi_t(L), s_t(L))^{\top} \in
C ([0,T]; D\left(\mathcal{A}\right)),$$of system \eqref{2bbm_nonhomo_2} which verifies \eqref{equa2}.
\end{theorem}
\begin{proof}
Let $ T  >  0 $ and $ \tau \in [0,T]$. By semigroup theory, precisely, from Theorem \ref{ex-hom}, the unique solution $W$ of system  \eqref{sistema3} is given by 
\begin{align*}
W(t)=S(\tau-t)W^{\tau},
\end{align*}
where $S(\cdot)$ is the semigroup generate by $\mathcal{A}$. Moreover, there exists $C_T  >  0$, such that
\begin{align}\label{deestiWparariesz}
\Vert W(t) \Vert_{\mathcal{H}} \leq C_T\Vert W^{\tau} \Vert_{\mathcal{H}}, \ \ \forall t \in [0, \tau].
\end{align}

Let us define a functional $\Delta$ given by the right-hand side of \eqref{equa2}, that is
\begin{align*}
\Delta\left(W^{\tau} \right) =& \left\langle V_0^1, W(0)\right\rangle_{ \mathcal{H}'\times\mathcal{H}}  \\&+\displaystyle\int_{0}^{\tau} \, \left \langle \left(3h_1(t), 3h_2(t, h_3(t)) \right), \left(\chi(t), \eta(t), \Theta(t) \right)\right\rangle_{\left[\left(  H_*^1(0,L) \right)^{*}, H_*^1(0,L) \right]^3}dxdt \\
&+\left((3ku_1(t), 3bu_2(t), bu_3(t)),  1_{(0, \tau)}(\chi(L, t), \eta(L, t), \Theta(L, t) )\right)_{[L^2(0, T)]^3}  \\
&+ \left( \left(3kw_t(L, 0), 3b\xi_t(L, 0), bs_t(L, 0)\right), \left( \chi(L, 0), \eta(L, 0), \Theta(L, 0) \right)\right)_{\mathbb{R}^3},
\end{align*}
where
\begin{align*}
V_0^1= ( 3 \rho_1 w_1,-3 \rho_1 w_0, 3 \rho_2 \xi_1,-3 \rho_2 \xi_0,
 \rho_2 s_1, -\rho_2 s_0,
 -3kw(L, 0),
 -3b\xi(L, 0),
 -bs(L, 0))^{\top}.
\end{align*}
Note that $\Delta$ is linear. Moreover, from \eqref{deestiWparariesz} and the Cauchy-Schwarz inequality, it follows that
\begin{equation*}
\begin{split}
 |\Delta\left(W^{\tau}\right)| \leq&\Vert V^1_0 \Vert_{\mathcal{H}} \Vert W(0) \Vert_{\mathcal{H}} + C_T \Vert W(t) \Vert_{\mathcal{H}} \Vert(u_1, u_2, u_3) \Vert_{[L^2(0,T)]^3}\\
 & + C_T \Vert W(t) \Vert_{C\left([0, \tau]; \mathcal{H}\right)} \Vert(h_1, h_2, h_3) \Vert_{ \left(L^2\left(0,T;  \left( H_*^1(0,L) \right)^{*}\right)\right)^3}\\
 & + C_T \Vert W(t) \Vert_{C\left([0, \tau]; \mathcal{H}\right)} \Vert\left(w_t(L, 0), \xi_t(L, 0), s_t(L, 0)\right) \Vert_{ \mathbb{R}^3}\\
  \leq &C_T \left( \Vert V^1_0 \Vert_{\mathcal{H}} + \Vert(u_1, u_2, u_3) \Vert_{[L^2(0,T)]^3} + \Vert(h_1, h_2, h_3) \Vert_{ \left(L^2\left(0,T;  \left( H_*^1(0,L) \right)^{*}\right)\right)^3} \right. \\
 & \left. + \Vert\left(w_t(L, 0), \xi_t(L, 0), s_t(L, 0)\right) \Vert_{ \mathbb{R}^3}    \right) \Vert W^{\tau} \Vert_{\mathcal{H}}.
\end{split}
\end{equation*}Hence,  we obtain that $\Delta \in \mathcal{L}(D\left(\mathcal{A}\right);\mathbb{R}). $ Thus, from the Riesz representation theorem, we obtain the existence and uniqueness of $U \in D\left(\mathcal{A}\right) $  satisfying \eqref{equa2}. Finally, the fact that $U \in C\left([0, T] ; D\left(\mathcal{A}\right) \right)$ is analogous to performed in \cite[Lemma 3.2.]{CF2019}, so we omit the details.
\end{proof} 

\section{Controllability properties} \label{sec:3}

In this section, we investigate the boundary controllability properties of system \eqref{2bbm}. As mentioned in the introduction, the controllability problem is closely related to establishing an observability inequality. Accordingly, we divide this section into two parts. First, we provide a characterization of the controllability of the system under consideration and present a corresponding observability inequality for the adjoint system. Finally, using this inequality, we prove the main result of the article.

\subsection{Observability inequality}

We begin with the following characterization of a control that drives system \eqref{2bbm} to zero. Such results are well-established for dispersive systems\footnote{See, for example, \cite{BAPA,M}.} and in a general context, for details see \cite{Komornik}. In our context, the following lemma gives an equivalent condition for the exact controllability property.

\begin{lemma}%\label{charac} 
The initial data  $U_0 \in \mathcal{H}'$
is controllable to zero in time $T > 0$ with controls $ u_i \in  L^2(0,T),$ for $i=1,2,3  $, if and only if
\begin{equation}\label{equiv}
\begin{split}
 &- \left\langle  \left(\begin{array}{c}
3 \rho_1 w_1 \\ -3 \rho_1 w_0 \\ 3 \rho_2 \xi_1\\-3 \rho_2 \xi_0 \\
 \rho_2 s_1 \\ -\rho_2 s_0\\
 -3kw(L, 0)\\
 -3b\xi(L, 0)\\
 -bs(L, 0)\end{array}\right), \left(\begin{array}{c}
 \chi(0)\\ \chi_t(0) \\ \eta(0) \\ \eta_t(0)\\ \Theta(0) \\ \Theta_t(0) \\\chi_t(L, 0)\\ \eta_t(L, 0)\\ \Theta_t(L, 0) \end{array}\right)  \right\rangle \\
&-\left( \left(3kw_t(L, 0), 3b\xi_t(L, 0) bs_t(L, 0), \left( \chi(L, 0), \eta(L, 0), \Theta(L, 0) \right)\right)\right)_{\mathbb{R}^3}  \\
& =  3k\int^T_0 u_1(t) \chi(L, t)dt + 3b\int^T_0 u_2(t)\eta(L, t) dt+ b\int^T_0 u_3(t)\Theta(L, t) dt.
\end{split}
\end{equation}
Here, $\left\langle \cdot , \cdot\right\rangle$ is the duality of $\mathcal{H}'\times\mathcal{H}$ and $(\chi,\eta,\Theta)$ is  any solution of the following adjoint system
\begin{equation}\label{back}
\left\{\begin{array}{ll} \rho_{1}\chi_{tt}-k\left(  \chi_{x}+\eta + \Theta\right)  _{x}   =0, &x\in(0,L),\,\,\, 
t>0,\\   
\rho_{2}\eta_{tt}-b\eta_{xx}+k\left(  \chi_{x}+\eta + \Theta\right)   
=0,  &x\in(0,L),\,\,\, 
t>0,\\
\rho_{2}\Theta_{tt}-b\Theta_{xx}+3k\left(  \chi_{x}+\eta + \Theta\right)  
+\gamma \Theta=0, & x\in(0,L),\,\,\, t>0,\\
\chi(0, t)=\eta(0, t)= \Theta(0, t) =0, & t>0,\\
\chi_{tt}(L, t)+\chi_x(L,t)+\eta(L,t)+\Theta(L, t)=0, &
t>0,\\
\eta_{tt}(L, t)+\Theta_{x}(L, t)= 0, &
t>0,\\
\Theta_{tt}(L, t) + \Theta_x(L, t) =0, &
t>0,\\
\left(  \chi, \eta, \Theta \right)  \left( x,T \right)  =\left(\chi^{T}_0, \eta^T_0, \Theta_0^T\right)(x), & x\in(0,L),\\
\left( \chi_{t}, \eta_{t}, \Theta_{t}\right)  \left(  x,T\right)  =\left(\chi^{T}_1, \eta^T_1, \Theta_1^T\right)(x), & x\in(0,L),\end{array}\right.
\end{equation}
with $(\chi_0^{\tau}, \chi_1^{\tau}, \eta_0^{\tau}, \eta_1^{\tau}, \Theta_0^{\tau}, \Theta_1^{\tau}, \chi_1^{\tau}(L), \eta_1^{\tau}(L), \Theta_1^{\tau}(L))^{\top}  \in \mathcal{H}$.
\end{lemma}
\begin{proof}The relation \eqref{equiv} is obtained first, by multiplying the equations $\eqref{2bbm}_1$, $\eqref{2bbm}_2$, and $\eqref{2bbm}_3$ by $3\chi$, $3\eta$ and $\Phi$, respectively, where $(\chi,\eta,\phi)$ is solution of \eqref{back}. So, integrating by parts and using the boundary conditions of \eqref{2bbm}, the relation \eqref{equiv} follows.
\end{proof}

Since we are using the control duality theory of Dolecki and Russell \cite{DoRu} in the set-up of Lions \cite{Lions1988SIAM}, relation \eqref{equiv} may be seen as an optimality condition for the critical points of the functional
$\mathcal{J}: \mathcal{H}\to \mathbb{R}$, given by 
$$\mathcal{J}(W_0)=\frac{1}{2} \int^T_0 \left( \left|\chi(L, t)\right|^2dt +  \left|\eta(L, t)\right|^2 + \left|\Theta(L, t)\right|^2 \right)dt+\left\langle W_0,U_0\right\rangle_{\mathcal{H}\times\mathcal{H}'},$$
where $W=\left( \chi, \chi_t, \eta, \eta_t, \Theta, \Theta_t,  \chi_t(L, t), \eta_t(L, t), \Theta_t(L, t) \right)^{\top} $ is solution of the adjoint system \eqref{back}. 

With this in hand, our task is to prove the existence of a minimizer for $\mathcal{J}$. It is well known that the existence of a minimizer to the functional $\mathcal{J}$ follows from the following observability inequality
\begin{equation}\label{DESIIMPORTANTE}
 \Vert W(0) \Vert^2_{\mathcal{H}} \leq C \int^T_0 \left( \left|\chi(L, t)\right|^2dt +  \left|\eta(L, t)\right|^2 + \left|\Theta(L, t)\right|^2 \right)dt,  
\end{equation}
for any  $$(\chi_0^{\tau}, \chi_1^{\tau}, \eta_0^{\tau}, \eta_1^{\tau}, \Theta_0^{\tau}, \Theta_1^{\tau}, \chi_1^{\tau}(L), \eta_1^{\tau}(L), \Theta_1^{\tau}(L))^{\top}  \in D(\mathcal{A}^*).$$ Here, $W=\left( \chi, \chi_t, \eta, \eta_t, \Theta, \Theta_t,  \chi_t(L, t), \eta_t(L, t), \Theta_t(L, t) \right)^{\top} $ is solution of the adjoint system \eqref{back}. Now, our efforts are to prove that this inequality holds. Before proving the inequality \eqref{DESIIMPORTANTE}, let us prove intermediary inequalities related to the traces of the solutions of \eqref{abs}. The first one can be read as follows.

\begin{proposition}%\label{firstmainresult1}
Let $U=(w,w_t, \xi,\xi_t, s,s_t, w_t(L,t), \xi_t(L,t), s_t(L,t))^{\top}$ the solution of the system \eqref{abs}. Then, for any $T>0$ and $s \in (\frac{1}{2}, 1)$, there exist a positive constant $C$ such that the following estimate hold
\begin{equation}\label{primerdesiobserder}
\begin{split}
\Vert U(t) \Vert^2_{\mathcal{H}} \leq &C \int^T_0 \left( \left|w_t(L, t)\right|^2dt +  \left|\xi_t(L, t)\right|^2 + \left|s_t(L, t)\right|^2 \right)dt  \\
&+C\Vert (w, \xi, s) \Vert^2_{L^{\infty}\left(0, T; [H^{s}(0, L)]^3\right)}.
\end{split}
\end{equation}
\end{proposition}
\begin{proof}
In order to obtain estimate \eqref{primerdesiobserder}, we multiply the first equation in \eqref{2bbm-hom} by $3x w_x$ and integrate by parts on $(0, L)\times (0, T) $ to obtain
\begin{equation}\label{ecua1_fordesiguaob1}
\begin{split}
 0 =& \int^L_0 \int^T_0 \left( \rho_{1}w_{tt}-k\left(  w_{x}+\xi+s\right)_{x} \right)3x w_x dt dx \\
  = &3 \rho_1 \int^L_0\left[ x w_tw_x \right]_0^Tdx -3k \int^T_0 L \left(  w_{x}(L, t)+\xi(L, t)+s(L, t)\right)w_{x}(L, t)dt  \\
 & +3k \int^L_0 \int^T_0 \left(  w_{x}+\xi+s\right)x w_{xx}dt dx + 3k \int^L_0 \int^T_0\left(  w_{x}+\xi+s\right)^2 dt dx  \\
 & -3k \int^L_0 \int^T_0\left(  w_{x}+\xi+s\right)\left( \xi+s\right) dt dx -3\rho_1 \int^L_0 \int^T_0 x w_t w_{xt} dt dx. 
 \end{split}
\end{equation}
Since 
\begin{align*}
\int^L_0 \int^T_0 x w_t w_{xt} dt dx = \frac{L}{2} \int^T_0w^2_t(L, t)dt  - \frac{1}{2} \int^L_0 \int^T_0w^2_tdt dx,
\end{align*}
by substituting  the above identity in \eqref{ecua1_fordesiguaob1}, we get
\begin{equation}\label{ecua1_1_fordesiguaob1}
\begin{split}
 &3 \rho_1 \int^L_0\left[ x w_tw_x \right]_0^Tdx - \frac{3 \rho_1 L}{2} \int^T_0w^2_t(L, t)dt + \frac{3 \rho_1}{2} \int^L_0 \int^T_0w^2_tdt dx \\
& -3k \int^T_0 L \left(  w_{x}(L, t)+\xi(L, t)+s(L, t)\right)w_{x}(L, t)dt +3k \int^L_0 \int^T_0 \left(  w_{x}+\xi+s\right)x w_{xx}dt dx   \\
& + 3k \int^L_0 \int^T_0\left(  w_{x}+\xi+s\right)^2 dt dx -3k \int^L_0 \int^T_0\left(  w_{x}+\xi+s\right)\left( \xi+s\right) dt dx=0.
\end{split}
\end{equation}
Now, we multiply  the second equation in \eqref{2bbm-hom} by $3x \xi_x$, the third
one by $x s_x$, integrate by parts on $(0, L)\times (0, T) $ and we have that
\begin{equation}\label{ecua2_fordesiguaob1}
\begin{split}
 0 =& \int^L_0 \int^T_0 \left(  \rho_{2}\xi_{tt}-b\xi_{xx}+k\left(  w_{x}+\xi+s\right)\right)3x \xi_x dt dx \\
  = &3 \rho_2 \int^L_0\left[ x \xi_t \xi_x \right]_0^Tdx -\frac{3 \rho_2 L}{2}  \int^T_0 \xi^2_{t}(L, t)dt + \frac{3 \rho_2 }{2}  \int^L_0\int^T_0 \xi^2_{t}dt dx \\
 & +3k \int^L_0 \int^T_0 \left(  w_{x}+\xi+s\right)x \xi_{x}dt dx - \frac{3 b L}{2}  \int^T_0 \xi^2_{x}(L, t)dt +  \frac{3 b }{2}  \int^L_0\int^T_0 \xi^2_{x}dt dx,
\end{split}
\end{equation}
and 
\begin{equation}\label{ecua3_fordesiguaob1}
\begin{split}
 0 = &\int^L_0 \int^T_0 \left(  \rho_{2}s_{tt}-bs_{xx}+3k\left(  w_{x}+\xi+s\right) \right)x s_x dt dx \\
 = & \rho_2 \int^L_0\left[ x s_t s_x \right]_0^Tdx -\frac{ \rho_2 L}{2}  \int^T_0 s^2_{t}(L, t)dt + \frac{\rho_2 }{2}  \int^L_0\int^T_0 s^2_{t}dt dx \\
 & +3k \int^L_0 \int^T_0 \left(  w_{x}+\xi+s\right)x s_{x}dt dx - \frac{ b L}{2}  \int^T_0 s^2_{x}(L, t)dt + \frac{ \gamma L}{2}  \int^T_0 s^2(L, t)dt\\
 &+  \frac{ b }{2}  \int^L_0\int^T_0 s^2_{x}dt dx -  \frac{ \gamma L}{2}  \int^L_0\int^T_0 s^2dt dx,
 \end{split}
\end{equation}
respectively.

A  straightforward computation shows that
\begin{equation}\label{I_0pri}
\begin{split}
 &3k \int^L_0 \int^T_0 \left(  w_{x}+\xi+s\right)\left(  xw_{xx}+x\xi_x+xs_x\right)dt dx \\
 &= \frac{3kL}{2} \int^T_0 \left(  w_{x}(L, t)+\xi(L, t)+s(L, t)\right)^2dt -\frac{3k}{2}\int^L_0 \int^T_0 \left(  w_{x}+\xi+s\right)^2 dt dx. 
 \end{split}
\end{equation}
Thus, adding the identities \eqref{ecua1_1_fordesiguaob1}-\eqref{ecua3_fordesiguaob1} and by taking into account \eqref{I_0pri}, it follows that
\begin{equation}\label{total_suma_some}
\begin{split}
&\frac{1}{2} \int^T_0 \Vert U(t) \Vert^2_{\mathcal{H}}dt \\
&= \frac{1}{2}\int^T_0 \left( (3 \rho_1L+3k)w^2_{t}(L, t) +(3 \rho_2L+3b)\xi^2_{t}(L, t)+ (\rho_2L+b)s^2_{t}(L, t)\right)dt  \\
&- \int^L_0\left[3 \rho_1 x w_t w_x +3 \rho_2 x \xi_t \xi_x +  \rho_2 x s_t s_x \right]_0^Tdx + \frac{ b L}{2}  \int^T_0 \left( 3\xi^2_{x}(L, t)+s^2_{x}(L, t)\right)dt  \\
& + 3k \int^L_0 \int^T_0\left(  w_{x}+\xi+s\right)\left( \xi+s\right) dt dx +\frac{3kL}{2}\int^T_0  \left(  w_{x}(L, t)+\xi(L, t)+s(L, t)\right)w_{x}(L, t)dt  \\
& -\frac{3kL}{2}\int^T_0  \left(  w_{x}(L, t)+\xi(L, t)+s(L, t)\right)\left( \xi(L, t)+s(L, t)\right)dt\\
&- \frac{ \gamma L}{2}  \int^T_0 s^2(L, t)dt + \gamma L\int^L_0\int^T_0 s^2dt dx.
\end{split}
\end{equation}

We need to analyse each term of the RHS of \eqref{total_suma_some}. First, applying the Young inequality,  we have that
\begin{equation}\label{EQ1P}
\begin{split}
 3k \int^L_0 \left(  w_{x}+\xi+s\right)\left( \xi+s\right)  dx \leq  &\varepsilon \left( 3 k\left\|(w_x+\xi+s)(\cdot, t)\right\|^2 + 3 b\left\|\xi_x(\cdot, t) \right\|^2 + b\left\|s_x(\cdot, t) \right\|^2\right) \\
 &+ C_{\varepsilon,k} \left\|(\xi, s)(\cdot, t) \right\|_{[H^s(0, L)]^2}^2 ,  
 \end{split}
\end{equation}
for any  $s \geq 0$ and  $\varepsilon>0$. On the other hand, using again the Young inequality and  from the Sobolev embedding\footnote{See, for instance, \cite{DPV,LiMa}.}, we obtain a  positive constant $C_{\varepsilon_1,k, b, L}$, such that 
\begin{equation}\label{EQ2P}
\begin{split}
 \frac{1}{2}&\left \vert\left( bL\left( 3\xi^2_{x}(L, t)+s^2_{x}(L, t)\right) + 3kL  \left(  w_{x}(L, t)+\xi(L, t)+s(L, t)\right)w_{x}(L, t) \right. \right. \\
&\left.\left.-3kL  \left(  w_{x}(L, t)+\xi(L, t)+s(L, t)\right)\left(  \xi(L, t)+s(L, t)\right) \right) \right \vert \\
\leq&  \varepsilon_1 \left( 3 k\left\|(w_x+\xi+s)(\cdot, t)\right\|^2 + 3 b\left\|\xi_x(\cdot, t) \right\|^2 + b\left\|s_x(\cdot, t) \right\|^2\right) \\
&+ C_{\varepsilon_1,k, b, L}\left\| (w, \xi, s)(\cdot, t) \right\|_{[H^s(0, L)]^3}^2,
 \end{split}
\end{equation}
for any $\varepsilon_1>0$ and for $s \in (\frac{1}{2}, 1)$. Now, for all $t \in [0, T]$ and applying once more Young's inequality,  we get that
\begin{equation}\label{EQ3P}
\begin{split}
\int^L_0\left \vert3 \rho_1 x w_t w_x +3 \rho_2 x \xi_t \xi_x +  \rho_2 x s_t s_x\right\vert &dx  \leq \varepsilon_2 \left( 3\rho_1 \Vert w_t(\cdot, t) \Vert^2 + 3\rho_2 \Vert \xi_t(\cdot, t) \Vert^2 \right.\\&\left.+ \rho_2 \Vert s_t(\cdot, t) \Vert^2 \right)+ C_{\varepsilon_2,\rho_1, \rho_2, L}\left\| (w, \xi, s)(\cdot, t) \right\|_{[H^s(0, L)]^3}^2, 
 \end{split}
\end{equation}
for any $\varepsilon_2>0$, $s \in (\frac{1}{2}, 1)$ and some constant $C_{\varepsilon_2,\rho_1, \rho_2, L}>0$. 

With these estimates in hand, let us replace the inequalities \eqref{EQ1P}-\eqref{EQ3P} into \eqref{total_suma_some} and pick $\varepsilon=\varepsilon_1=\varepsilon_2$, it follows that
\begin{equation*}
\begin{split}
\frac{1}{2} \int^T_0 \Vert U(t) \Vert^2_{\mathcal{H}}dt \leq& 2\varepsilon \int^T_0 \Vert U(t) \Vert^2_{\mathcal{H}}dt + C\int^T_0 \left( w^2_{t}(L, t) +\xi^2_{t}(L, t)+ s^2_{t}(L, t)\right)dt\\
& +C_{\varepsilon}\int^T_0 \left\| (w, \xi, s)(\cdot, t) \right\|_{[H^s(0, L)]^3}^2dt.
 \end{split}
\end{equation*}
From estimative above and we can choose $\varepsilon$ small enough such that $0<\varepsilon<\frac{1}{4}$ to deduce
\begin{equation}\label{prefinalesti_impor0}
\begin{split}
\int^T_0 \Vert U(t) \Vert^2_{\mathcal{H}}dt \leq &C\int^T_0 \left( w^2_{t}(L, t) +\xi^2_{t}(L, t)+ s^2_{t}(L, t)\right)dt  \\
&+C\Vert (w, \xi, s) \Vert^2_{L^{\infty}\left(0, T; [H^{s}(0, L)]^3\right)}, 
 \end{split}
\end{equation}
for any $s \in (\frac{1}{2}, 1)$. Finally, using \eqref{identi_gisometr} and the inequality \eqref{prefinalesti_impor0}, we infer the estimative \eqref{primerdesiobserder}.  This finishes the proof.
\end{proof}

Let us now establish the second auxiliary inequality of this section.

\begin{theorem}\label{teoimportantepaperI}
Let $U=(w,w_t, \xi,\xi_t, s,s_t, w_t(L,t), \xi_t(L,t), s_t(L,t))^{\top}$ the solution of the system \eqref{abs}. Then,  there exists a positive constant $C$ such that the following estimate holds
\begin{align}\label{seconddesiobserder2}
\Vert U_0 \Vert^2_{\mathcal{H}} \leq C \int^T_0 \left( \left|w_t(L, t)\right|^2dt +  \left|\xi_t(L, t)\right|^2 + \left|s_t(L, t)\right|^2 \right)dt.
\end{align}
\end{theorem}

\begin{proof}
Let us argue by contradiction, following the so-called “compactness-uniqueness”
argument\footnote{For details of this argument see, for instance, \cite{Lions1988SIAM,ZUA}.}. Suppose that \eqref{seconddesiobserder2} does not hold. Thus,  there exists a sequence $\{U^n_0\}_{n\in\mathbb{N}} \in \mathcal{H}$, such that
\begin{equation}\label{rrrr}
1=\Vert U^n_0 \Vert^2_{\mathcal{H}} >n \int^T_0 \left( \left|w^n_t(L, t)\right|^2dt +  \left|\xi^n_t(L, t)\right|^2 + \left|s^n_t(L, t)\right|^2 \right)dt.
\end{equation}
Note que \eqref{rrrr} is equivalent to the next two assertions: 
\begin{align}\label{noriguaconver_1}
\Vert U^n(t) \Vert^2_{\mathcal{H}}=\Vert U^n_0 \Vert^2_{\mathcal{H}} = 1, \ \ \text{ for all } \ t \in [0,T],
\end{align}
and
\begin{align}\label{dericonver_2}
\int^T_0 \left( \left|w^n_t(L, t)\right|^2dt +  \left|\xi^n_t(L, t)\right|^2 + \left|s^n_t(L, t)\right|^2 \right)dt \longrightarrow 0, \ \text{ as} \ n \rightarrow \infty,
\end{align}
where the function $U^n=(w^n,w^n_t, \xi^n,\xi^n_t, s^n,s^n_t, w^n_t(L,t), \xi^n_t(L,t), s^n_t(L,t))^{\top}$ solves
\begin{equation}\label{abswithn}
\left\{
\begin{array}
[c]{l}
U^n_t(t) = \mathcal{A} U^n(t),\\
U^n(0) = U^n_0.
\end{array}\right.
\end{equation}

Observe that, from \eqref{noriguaconver_1} and the definition of the norm, we get
\begin{equation}\label{paracompac_1}
\begin{cases}% \begin{array}{lll}
(w^n, \xi^n, s^n)  & \text { bounded in }  L^{\infty}\left(0, T ;  [H^{1}(0, L)]^3\right) \\
(w^n_t, \xi^n_t, s^n_t) & \text { bounded in }  L^{\infty}\left(0, T; [L^{2}(0, L)]^3\right).
%\end{array}
\end{cases}
\end{equation}
%\begin{align}\label{paracompac_1}
%\Vert (w^n, \xi^n, s^n) \Vert^2_{L^{\infty}\left(0, T; [H^{1}(0, L)]^3\right)}+\Vert (w^n_t, \xi^n_t, s^n_t) \Vert^2_{L^{2}\left(0, T; [L^{2}(0, L)]^3\right)} \leq C.
%\end{align}
Since $H^{1}(0, L)\hookrightarrow H^s(0,L) \hookrightarrow L^{2}(0, L)$ and  the embedding $H^{1}(0, L)\hookrightarrow H^s(0,L)$ is compact, for all $s \in (\frac{1}{2}, 1)$, from \eqref{paracompac_1} and classical compactness results  \cite[Corollary 4]{SIMO}, we can extract a subsequence of $\{(w^n, \xi^n, s^n)\}_{n\in\mathbb{N}}$, still denoted by $\{(w^n, \xi^n, s^n)\}_{n\in\mathbb{N}}$, such that
\begin{align}\label{firsconverimp}
(w^n, \xi^n, s^n) \longrightarrow (w, \xi, s) \ \ \text{ strongly in } \ L^{\infty}\left(0, T; [H^{s}(0, L)]^3\right). 
\end{align}Thus, from \eqref{primerdesiobserder}, \eqref{dericonver_2} and \eqref{firsconverimp}, we deduce that $\{U^n\}_{n\in\mathbb{N}}$ is a Cauchy sequence in $L^{\infty}\left(0, T; \mathcal{H}\right)$.  Therefore, 
\begin{align*}
U^n \longrightarrow U \ \ \text{ strongly in } \ L^{\infty}\left(0, T; \mathcal{H}\right),
\end{align*}
and by equality \eqref{noriguaconver_1}, yields that
\begin{align}\label{nor1soluU}
\Vert U(t) \Vert^2_{\mathcal{H}} = 1, \ \ \text{ for all } \ t \in [0,T].
\end{align}
Also,
\begin{equation}\label{cerofronteraU}
\begin{split}
0& = \liminf_{n \rightarrow 0} \left\lbrace \int^T_0 \left( \left|w^n_t(L, t)\right|^2dt +  \left|\xi^n_t(L, t)\right|^2 + \left|s^n_t(L, t)\right|^2 \right)dt\right\rbrace \\
& \geq \int^T_0 \left( \left|w_t(L, t)\right|^2dt +  \left|\xi_t(L, t)\right|^2 + \left|s_t(L, t)\right|^2 \right)dt. 
\end{split}
\end{equation}
From \eqref{cerofronteraU} and the above convergences, passing the limit in \eqref{abswithn}, we have that the limit $U=(w,w_t, \xi,\xi_t, s,s_t, w_t(L,t), \xi_t(L,t), s_t(L,t))^{\top}$ satisfies the linear system
\begin{equation}\label{2bbm_nonhomo_Ulimi}
\left\{\begin{array}{ll} 
\rho_{1}w_{tt}-k\left(  w_{x}+\xi+s\right)_{x}   =0, &x\in(0,L),\,\,\, 
t \in (0, T),\\   
\rho_{2}\xi_{tt}-b\xi_{xx}+k\left(  w_{x}+\xi+s\right)   
=0,  &x\in(0,L),\,\,\, 
t\in (0, T),\\
\rho_{2}s_{tt}-bs_{xx}+3k\left(  w_{x}+\xi+s\right)  
+ \gamma s=0, & x\in(0,L),\,\,\, t\in (0, T),\\
w(0, t)=\xi(0, t)= s(0, t) =0, & t\in (0, T),\\
w_x(L,t)+\xi(L,t)+s(L, t)=0, &
t\in (0, T),\\
\xi_{x}(L, t)= 0, &
t\in (0, T),\\
s_x(L, t) =0, &
t\in (0,T).\end{array}\right.
\end{equation} 
Let $( \hat{w}, \hat{\xi}, \hat{s})$ the extension by zero of $(w, \xi, s)$, for $ x \in (-a,a)\setminus(0,L)$, where $ (0,L)\subset(-a,a)$ it is an interval. Then, $( \hat{w}, \hat{\xi}, \hat{s})$ solves, in $D^{'}((0, L)\times (0, T))$, the following system 
 
\begin{equation}\label{bbm_nonhomo_extensionessolu}
\left\{\begin{array}{ll} 
\rho_{1}\hat{w}_{tt}-k(  \hat{w}_{x}+\hat{\xi}+\hat{s})_{x}   =0, &x\in(0,L),\,\,\, 
t \in (0, T),\\   
\rho_{2}\hat{\xi}_{tt}-b\hat{\xi}_{xx}+k(  \hat{w}_{x}+\hat{\xi}+\hat{s})   
=0,  &x\in(0,L),\,\,\, 
t\in (0, T),\\
\rho_{2}\hat{s}_{tt}-b\hat{s}_{xx}+3k(  \hat{w}_{x}+\hat{\xi}+\hat{s})  
+ \gamma \hat{s}=0, & x\in(0,L),\,\,\, t\in (0, T),\\
\hat{w}(0, t)=\hat{\xi}(0, t)= \hat{s}(0, t) =0, & t\in (0, T),\\
\hat{w}_x(L,t)+\hat{\xi}(L,t)+\hat{s}(L, t)=0, &
t\in (0, T),\\
\hat{\xi}_{x}(L, t)= 0, &
t\in (0, T),\\
\hat{s}_x(L, t) =0, &
t\in (0,T) \\
( \hat{w}, \hat{\xi}, \hat{s}) =(0, 0, 0), & x \in (-a, 0] \cup [L, a), \ t\in (0, T),
\end{array}\right.
\end{equation} 
Then, by Holmgren’s uniqueness theorem, $$( \hat{w}, \hat{\xi}, \hat{s}) \equiv (0, 0, 0)  \text{ in } (-a, a)\times(0,T).$$ Hence, from \eqref{2bbm_nonhomo_Ulimi} and \eqref{bbm_nonhomo_extensionessolu}, we have that  $U=(w,w_t, \xi,\xi_t, s,s_t, w_t(L,t), \xi_t(L,t), s_t(L,t))^{\top}$ satisfies 
\begin{equation}\label{2bbm_nonhomo_Ulimi_deridoscero}
\left\{\begin{array}{ll} 
k\left(  w_{x}+\xi+s\right)_{x}   =0, &x\in(0,L),\,\,\, 
t \in (0, T),\\   
b\xi_{xx}-k\left(  w_{x}+\xi+s\right)   
=0,  &x\in(0,L),\,\,\, 
t\in (0, T),\\
bs_{xx}-3k\left(  w_{x}+\xi+s\right)  
-\gamma s=0, & x\in(0,L),\,\,\, t\in (0, T),\\
w(0, t)=\xi(0, t)= s(0, t) =0, & t\in (0, T),\\
w_x(L,t)+\xi(L,t)+s(L, t)=0, &
t\in (0, T),\\
\xi_{x}(L, t)= 0, &
t\in (0, T),\\
s_x(L, t) =0, &
t\in (0,T).\end{array}\right.
\end{equation}
Thanks to the uniqueness of the solution to the previous system established by Lemma \ref{0inresolv}, the only solution to \eqref{2bbm_nonhomo_Ulimi_deridoscero} is U = 0. However, this contradicts \eqref{nor1soluU}. Therefore, \eqref{seconddesiobserder2} must hold, which completes the proof of the proposition.
\end{proof}

%\begin{remark}\label{directinequali}
%From \eqref{identi_gisometr},  we directly verify that 
%\begin{align}\label{direcdirec}
%  \int^T_0 \left( \left|w_t(L, t)\right|^2dt +  \left|\xi_t(L, t)\right|^2 + \left|s_t(L, t)\right|^2 \right)dt \leq C \Vert U_0 \Vert^2_{\mathcal{H}}.  
%\end{align}
%\end{remark}

We are now in a position to prove the observability inequality.

\begin{theorem}%\label{teoprindesiusados}
For any $U_0 \in\mathcal{H}$, let $U=(w,w_t, \xi,\xi_t, s,s_t, w_t(L,t), \xi_t(L,t), s_t(L,t))^{\top}$ the solution of the system \eqref{abs}. Then,  there exists a positive constant $C$ such that the following estimate holds
\begin{align}\label{thirdesiobserder3}
\Vert U_0 \Vert^2_{\mathcal{H}} \leq C \int^T_0 \left( \left|w(L, t)\right|^2dt +  \left|\xi(L, t)\right|^2 + \left|s(L, t)\right|^2 \right)dt.
\end{align}
\end{theorem}
\begin{proof}
In order to proof the inequality \eqref{thirdesiobserder3}, we consider
 the problem
 \begin{equation}\label{abssombrero}
\left\{
\begin{array}
[c]{l}
\tilde{U}_t = \mathcal{A} \tilde{U}\\
\tilde{U}(0) = \tilde{U}_0,
\end{array}\right.
\end{equation}where $\tilde{U}_0=\mathcal{A}^{-1}U_0$. It is important to point out that the existence of $\mathcal{A}^{-1}$ is guaranteed by Lemma \ref{0inresolv}. Thus, from Theorem \ref{teoimportantepaperI}, the solution  $$\tilde{U}=(\tilde{w},\tilde{w}_t, \tilde{\xi}, \tilde{\xi}_t, \tilde{s}, \tilde{s}_t, \tilde{w}_t(L,t), \tilde{\xi}_t(L,t), \tilde{s}_t(L,t))^{\top}$$ of the system \eqref{abssombrero} satisfies
\begin{align}\label{tresdesiobserderprima}
\Vert \tilde{U}_0 \Vert^2_{\mathcal{H}} \leq C \int^T_0 \left( \left|\tilde{w}_t(L, t)\right|^2dt +  \left|\tilde{\xi}_t(L, t)\right|^2 + \left|\tilde{s}_t(L, t)\right|^2 \right)dt,
\end{align}for some constant $C>0$.

On the other hand, it is straightforward to see that $\tilde{U}_t = U$. Then, we have that
\begin{align}\label{equalderiandwxis}
(\tilde{w}_t, \tilde{\xi}_t, \tilde{s}_t)=(w, \xi, s).
\end{align}Finally, from \eqref {tresdesiobserderprima}, \eqref{equalderiandwxis} and  by Riesz representation theorem, we deduce that
\begin{align*}
 \Vert U_0 \Vert^2_{\mathcal{H}}\leq& C\Vert \tilde{U}_0 \Vert^2_{\mathcal{H}} \\
 \leq & C \int^T_0 \left( \left|\tilde{w}_t(L, t)\right|^2dt +  \left|\tilde{\xi}_t(L, t)\right|^2 + \left|\tilde{s}_t(L, t)\right|^2 \right)dt \\
=& C \int^T_0 \left( \left|w(L, t)\right|^2dt +  \left|\xi(L, t)\right|^2 + \left|s(L, t)\right|^2 \right)dt,
\end{align*}
and the proof is achieved. 
\end{proof}

Due to the previous proposition, the following consequence holds.

\begin{corollary}\label{finaldesforadforcontrolteo}
Let $W=\left( \chi, \chi_t, \Theta, \Theta_t, \eta, \eta_t, \chi_t(L, t), \eta_t(L, t), \Theta_t(L, t) \right)^{\top} $ is solution of the adjoint system \eqref{back}.  Then, there exists a 
constant $C > 0$, such that the following observability inequality holds 
\begin{equation}\label{DESIIMPORTANTE_colorarioprobado}
 \Vert W(0) \Vert^2_{\mathcal{H}} \leq C \int^T_0 \left( \left|\chi(L, t)\right|^2dt +  \left|\eta(L, t)\right|^2 + \left|\Theta(L, t)\right|^2 \right)dt.  
\end{equation}
\end{corollary}

\subsection{Proof of Theorem \ref{controlteorfinalnodemos}}
As a direct consequence of the observability inequality established in Corollary \ref{finaldesforadforcontrolteo}, combined with the HUM method from control theory, we obtain the controllability of the laminated system with Ventcel boundary conditions. In fact, let us consider the homogeneous system \eqref{abs} with initial data $U_0\in H'$, and let $U$ be its corresponding solution. To evaluate this initial state in terms of boundary observables, we introduce the functional
 $\mathcal{J}: \mathcal{H}\to \mathbb{R}$, given by 
$$\mathcal{J}(W_0)=\frac{1}{2} \int^T_0 \left( \left|\chi(L, t)\right|^2dt +  \left|\eta(L, t)\right|^2 + \left|\Theta(L, t)\right|^2 \right)dt+\left\langle W_0,U_0\right\rangle_{\mathcal{H}\times\mathcal{H}'},$$
where $W=\left( \chi, \chi_t, \eta, \eta_t, \Theta, \Theta_t,  \chi_t(L, t), \eta_t(L, t), \Theta_t(L, t) \right)^{\top} $ is solution of the adjoint system. 

Note that $\mathcal{J}$ is continuous and convex. Moreover, the coercivity of functional $\mathcal{J}$ follows immediately from the observability inequality \eqref{finaldesforadforcontrolteo}. Thus, thanks to \cite[Corollary 3.23]{Brezis}, we have that $\mathcal{J}$ has a minimizer. Since $\mathcal{J}$ is strictly convex, it follows that the minimizer is unique. From these facts, defining $u_1=\chi(L, t)$, $u_2=\eta(L, t)$ and $u_3=\Theta(L, t)$ the null controllability holds true.\qed

\section{Final comments}\label{sec:4}
In this work, we address the boundary controllability of a laminated beam governed by a system of three coupled equations, subject to dynamic Venttsel-type boundary conditions. To establish the required observability inequality, we employ the multiplier technique (cf. \cite{Komornik}) along with the Hilbert Uniqueness Method (cf. \cite{Lions1988SIAM}). The main challenge stems from the need to control a system of three beam equations using three boundary inputs. As this is the first study tackling this specific problem, our results pave the way for further investigations in both the control and stabilization of such systems. Let us point out some of them.
\begin{itemize}
\item \textit{Less controls}: A natural question arising from our analysis is whether it is possible to reduce the number of boundary controls. We believe that our current approach is optimal in the sense that it achieves controllability of the three equations using three boundary inputs. Nonetheless, we conjecture that, by applying Carleman estimates to the linear operator defined in \eqref{opae2}, it may be feasible to eliminate one of the boundary controls. However, this remains an open problem at present.
\item \textit{Rapid stabilization}: It is well known in control theory that 
\begin{equation*}
\text{observability}\Longleftrightarrow\text{controllability}
\quad \text{and} \quad \text{observability}\Longrightarrow\text{rapid stabilizability}
\end{equation*}
We believe that an application of an Ingham-type theorem (cf. \cite{Ingham-1936}) to the system \eqref{2bbm} could lead to a sharper observability inequality, which in turn would support the derivation of faster stabilization results. This remains another important issue to be explored.
\end{itemize}

\vspace{0.2cm}

%\subsection*{Acknowledgements}
%\subsection*{Funding} 
\subsection*{Data Availability} It does not apply to this article as no new data were created or analyzed in this study.
\subsection*{Conflict of interest} This work does not have any conflicts of interest.
%The first author was supported by Faperj (Brazil). The second author was partially supported by CNPq (Brazil).

\end{document}